\documentclass[12pt]{article}
\usepackage{hyperref}
\usepackage{stmaryrd}
\usepackage{a4}
\usepackage{amsthm}
\usepackage{amsmath}
\usepackage{amssymb}
\usepackage{amsfonts}
\usepackage{cite}
\usepackage{epsfig}

\newtheorem{theorem}{Theorem}
\newtheorem{proposition}[theorem]{Proposition}

\DeclareMathOperator{\Aut}{Aut}
\newcommand{\NN}{\mathbb{N}}
\newcommand{\dd}{\;\mathrm{d}}

\begin{document}
\title{No additional tournaments are quasirandom-forcing\thanks{The first, third and seventh authors were supported by the European Research Council (ERC) under the European Union's Horizon 2020 research and innovation programme (grant agreement No 648509). The second, third and sixth authors were supported by the MUNI Award in Science and Humanities of the Grant Agency of Masaryk University. The fourth author was supported by CAPES. This publication reflects only its authors' view; the European Research Council Executive Agency is not responsible for any use that may be made of the information it contains.}}

\date{}
\author{Robert Hancock \and Adam Kabela \and Daniel Kr\'al' \and Ta\'isa Martins \and Roberto Parente \and Fiona Skerman \and Jan Volec}

\newcommand{\Addresses}{{
  \bigskip
  \footnotesize

\noindent Robert Hancock, \textsc{Institut f\"ur Informatik, Heidelberg University, Im Neuenheimer Feld 205, 69120, Heidelberg, Germany.}
\textit{Previous affiliation:} \textsc{Faculty of Informatics, Masaryk University, Botanick\'a 68A, 602 00 Brno, Czech Republic.}
\par\nopagebreak \textit{E-mail:} \texttt{hancock@informatik.uni-heidelberg.de}

\medskip

\noindent Adam Kabela, \textsc{Faculty of Applied Sciences, University of West Bohemia, Pilsen, Czech Republic.}
\textit{Previous affiliation:} \textsc{Faculty of Informatics, Masaryk University, Brno, Czech Republic.}
\par\nopagebreak \textit{E-mail:} \texttt{kabela@kma.zcu.cz}

\medskip

\noindent Daniel Kr\'al', \textsc{Faculty of Informatics, Masaryk University, Botanick\'a 68A, 602 00 Brno, Czech Republic.}
\textit{Previous affiliation:} \textsc{Mathematics Institute, DIMAP and Department of Computer Science, University of Warwick, Coventry CV4 7AL, UK.}
\par\nopagebreak \textit{E-mail:} \texttt{dkral@fi.muni.cz}

\medskip

\noindent Ta\'isa Martins, \textsc{Instituto de Matem\'atica e Estat\'istica, Universidade Federal Fluminense, Niter\'oi, Brazil.}
\par\nopagebreak \textit{E-mail:} \texttt{tlmartins@id.uff.br}

\medskip

\noindent Roberto Parente, \textsc{Instituto de Computa\c c\~ao, Universidade Federal da Bahia, Salvador, Brazil.}
\par\nopagebreak \textit{E-mail:} \texttt{roberto.parente@ufba.br}

\medskip

\noindent Fiona Skerman, \textsc{Mathematics Institute, Uppsala University, Uppsala, Sweden.}
\textit{Previous affiliation:} \textsc{Faculty of Informatics, Masaryk University, Botanick\'a 68A, 602 00 Brno, Czech Republic.}
\par\nopagebreak \textit{E-mail:} \texttt{fiona.skerman@math.uu.se}

\medskip

\noindent Jan Volec, \textsc{Department of Mathematics, Faculty of Nuclear Sciences and Physical Engineering, Czech Technical University in Prague, Trojanova 13, 120 00 Prague, Czech Republic.}
\textit{Previous affiliation:} \textsc{Faculty of Informatics, Masaryk University, Botanick\'a 68A, 602 00 Brno, Czech Republic.}
\par\nopagebreak \textit{E-mail:} \texttt{jan@ucw.cz}
}}

\maketitle
\begin{abstract}
A tournament $H$ is quasirandom-forcing if the following holds
for every sequence $(G_n)_{n\in\NN}$ of tournaments of growing orders:
if the density of $H$ in $G_n$ converges to 
the expected density of $H$ in a random tournament,
then $(G_n)_{n\in\NN}$ is quasirandom.
Every transitive tournament with at least $4$ vertices is quasirandom-forcing,
and Coregliano et al.~[Electron. J. Combin. 26 (2019), P1.44] showed that
there is also a non-transitive $5$-vertex tournament with the property.
We show that no additional tournament has this property.
This extends the result of Buci\'c et al.~[Combinatorica 41 (2021), 175--208] that
the non-transitive tournaments with seven or more vertices do not have this property.
\end{abstract}

\section{Introduction}
\label{sec:intro}

A combinatorial structure is said to be \emph{quasirandom}
if it has properties that a random structure would have asymptotically almost surely.
The notion of~\emph{quasi\-random graphs} goes back
to the works of R\"odl~\cite{Rod86}, Thomason~\cite{Tho87,Tho87b} and
Chung, Graham and Wilson~\cite{ChuGW89} from the 1980s.
There is a long series of results concerning quasirandomness of many other kinds of combinatorial structures,
for example 
groups~\cite{Gow08},
hypergraphs~\cite{ChuG90,ChuG91s,Gow06,Gow07,HavT89,KohRS02,NagRS06,RodS04},
permutations~\cite{ChaKNPSV20,Coo04,KraP13},
subsets of integers~\cite{ChuG92}, etc.
In the present short paper,
we consider quasirandomness of tournaments as studied in~\cite{BucLSS19,ChuG91,CorR17};
several equivalent definitions of this notion can be found in~\cite{ChuG91}.

One of the classical results on quasirandom graphs~\cite{ChuGW89,Rod86,Tho87} asserts that
an $n$-vertex graph with edge density $p$ is quasirandom if
it has $3\binom{n}{4}p^4+o(n^4)$ cycles of length four,
i.e., if the number of $4$-cycles is close to its expected value in a random graph with edge density~$p$.
Skokan and Thoma~\cite{SkoT04} showed that any complete bipartite graph $K_{a,b}$ with $a,b\ge2$ has the analogous property,
i.e., a graph is quasirandom if the number of copies of $K_{a,b}$ is close to its expected value in a random graph with the same edge density.
One of the major open problems in extremal combinatorics is the
Forcing Conjecture by Conlon, Fox and Sudakov~\cite{ConFS10} asserting that all bipartite graphs with a
cycle have this property.

We are interested in the same phenomenon for tournaments:
a tournament $H$ is \emph{quasirandom-forcing}
if the density of $H$ in $(G_n)_{n\in\NN}$ converging to the expected density of $H$ in a random tournament
is sufficient to guarantee the quasirandomness of the sequence.
In particular, if the density of $H$ converges to its expected density,
then the density of every tournament converges to its expected density in a random tournament.
Every transitive tournament $T_k$ with $k\ge 4$ vertices
is known to be quasirandom-forcing, see~\cite{CorR17} and \cite[Exercise 10.44]{Lov93}, and
Buci\'{c}, Long, Shapira and Sudakov~\cite{BucLSS19} observed that
every quasirandom-forcing tournament with seven or more vertices is transitive.
On the other hand, Coregliano, Parente and Sato~\cite{CorPS19} showed that
there is a non-transitive $5$-vertex tournament $F_5$ that is quasirandom-forcing;
the tournament $F_5$, which is called $T^{8}_{5}$ in~\cite{CorPS19}, is depicted in Figure~\ref{fig:F5}.
Our main result asserts that there is no quasirandom-forcing tournament
in addition to $T_k$, for $k\ge 4$, and $F_5$.

\begin{figure}
\begin{center}
\epsfbox{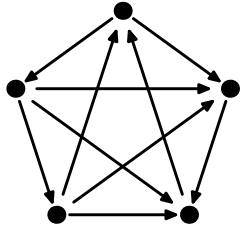}
\end{center}
\caption{The non-transitive tournament $F_5$ on $5$ vertices that is quasirandom-forcing.}
\label{fig:F5}
\end{figure}

The paper is structured as follows.
In Section~\ref{sec:prelim}, we recall from~\cite{BeiH65} classical results on the maximum numbers of cycles of length three and four in a tournament,
which rule out the existence of a strongly connected quasirandom-forcing tournament with at most $4$ vertices.
In Section~\ref{sec:general}, we first show that every non-transitive quasirandom-forcing tournament must be strongly connected,
hence we may focus on tournaments with $5$ and $6$ vertices only.
We next show that every quasirandom-forcing $6$-vertex tournament must be rigid and twin-free, which 
together with the results of Coregliano et al.~\cite{CorPS19} leaves a single $5$-vertex tournament
and exactly $14$ tournaments with $6$ vertices that are strongly connected and may be quasirandom-forcing.
We analyze all these $15$ tournaments in~Section~\ref{sec:constructions}.

\section{Preliminaries}
\label{sec:prelim}

In this section, we introduce notation and basic results used in the paper.
We write $[k]$ for the set $\{1,\ldots,k\}$.
A \emph{tournament} is a graph $G$
where each pair of vertices is joined by an edge oriented in one or the other direction;
we write $\lvert G\rvert$ for the number of vertices of $G$.
The \emph{adjacency matrix} of a tournament is the matrix $A$ with rows and columns indexed by the vertices of $G$ such that
its diagonal entries are zero, and $A_{uv}=1$ and $A_{vu}=0$ for every edge $uv$.
A tournament is \emph{rigid} if it has no non-trivial automorphism.
Two vertices $u$ and $v$ in a tournament are referred to as \emph{twins}
if every out-neighbor of $u$ possibly except for $v$ is an out-neighbor of $v$ and
every out-neighbor of $v$ possibly except for $u$ is an out-neighbor of $u$.
A tournament with no twins is said to be \emph{twin-free}.

If $G$ and $H$ are tournaments,
the \emph{density} of $H$ in $G$, which is denoted by $d(H,G)$,
is the probability that $\lvert H\rvert$ randomly chosen vertices of $G$ induce $H$;
if $\lvert H\rvert>\lvert G\rvert$, we set $d(H,G)=0$.
A sequence $(G_n)_{n\in\NN}$ of tournaments is \emph{quasirandom} if
\[\lim_{n\to\infty} d(H,G_n)=\frac{k!}{\lvert\Aut(H)\rvert}\cdot 2^{-\binom{k}{2}}\]
for every tournament $H$,
where $\Aut(H)$ is the group of automorphisms of $H$ (note that the right side of the expression
is the expected density of $H$ in a random tournament with $n\ge\lvert H\rvert$ vertices).
Finally, we say that a tournament $H$ is \emph{quasirandom-forcing}
if every sequence $(G_n)_{n\in\NN}$ of tournaments satisfying
\[\lim_{n\to\infty} d(H,G_n)=\frac{k!}{\lvert\Aut(H)\rvert}\cdot 2^{-\binom{k}{2}}\]
is quasirandom
(only sequences satisfying $|G_n|\to\infty$ as $n\to\infty$ are considered).
As we mentioned in Section~\ref{sec:intro},
every $k$-vertex transitive tournament $T_k$, for $k\ge 4$, is quasirandom-forcing, and
there is also a $5$-vertex strongly connected tournament that is quasirandom-forcing (this is the tournament $F_5$ depicted in Figure~\ref{fig:F5}).

We treat quasirandomness of tournaments in the language of theory of combinatorial limits,
which associates (convergent) sequences of combinatorial structures with analytic limit objects.
We refer the reader to the monograph by Lov\'asz~\cite{Lov12} for the treatment of the most studied case of graph limits,
which readily translates to the setting of tournament limits (see~\cite{DiaJ08,Tho18,ZhaZ20}).

We say that a sequence $(G_n)_{n\in\NN}$ of tournaments with $\lvert G_n\rvert$ tending to infinity
is \emph{convergent} if $d(H,G_n)$ converges for every tournament $H$.
A \emph{tournamenton} $W$ is a measurable function $[0,1]^2\to [0,1]$ such that $W(x,y)+W(y,x)=1$ for all $(x,y)\in [0,1]^2$.
The \emph{density} of a $k$-vertex tournament $H$ with vertices $v_1,\ldots,v_k$ in a tournamenton $W$,
which is denoted by $d(H,W)$, is
\begin{equation}\label{eq:d(H,W)}
d(H,W)=\frac{k!}{\lvert\Aut(H)\rvert}\;\int\limits_{[0,1]^k}\prod_{\overrightarrow{v_iv_j}\in E(H)}W(x_i,x_j)\dd x_1\cdots\dd x_k,
\end{equation}
where $E(H)$ is the set of (oriented) edges of $H$.
For every convergent sequence $(G_n)_{n\in\NN}$ of tournaments,
there exists a tournamenton $W$ such that
the limit density of each tournament $H$ in the sequence is equal to the density of $H$ in $W$;
we say that such $W$ is a \emph{limit} of the sequence $(G_n)_{n\in\NN}$ and
that the sequence $(G_n)_{n\in\NN}$ \emph{converges} to $W$.
Conversely, for every tournamenton $W$, there exists a sequence of tournaments that converges to $W$.

The definition of a quasirandom-forcing tournament translates to the limit setting as follows.

\begin{proposition}
\label{prop:qforce}
A tournament $H$ is quasirandom-forcing if
every tournamenton $W$ satisfying
\[d(H,W)=\frac{k!}{\lvert\Aut(H)\rvert}\cdot 2^{-\binom{k}{2}}\]
is equal to $1/2$ almost everywhere.
\end{proposition}

Proposition~\ref{prop:qforce} yields the following,
which was also noted at the end of Section~2 in~\cite{BucLSS19}.
We state the proposition in the language of combinatorial limits.

\begin{proposition}
\label{prop:qlarge}
Let $H$ be a tournament that is not transitive.
If there exists a tournamenton $W$ such that $W$ is not equal to $1/2$ almost everywhere and
\[d(H,W) \geq \frac{k!}{\lvert\Aut(H)\rvert}\cdot 2^{-\binom{k}{2}},\]
then $H$ is not quasirandom-forcing.
\end{proposition}

\begin{proof}
Let $W$ be the tournamenton given by the statement.
Let $T$ be the following tournamenton, which is a limit of a sequence of transitive tournaments:
\[T(x,y)=\begin{cases}
    1, & \mbox{if $x>y$,} \\
    1/2, & \mbox{if $x=y$, and} \\
    0, & \mbox{otherwise.}
    \end{cases}\]
Further, we define a $U_\alpha$ for $\alpha \in [0,1]$ as
\[U_\alpha(x,y)=\begin{cases}
    W(x,y), & \mbox{if $(x,y)\in [0,\alpha]^2$, and} \\
    T(x,y), & \mbox{otherwise.}
    \end{cases}\]
Observe that, for any $\alpha \in [0,1]$, $U_\alpha$ is not equal to $1/2$ almost everywhere.
Since the tournament $H$ is not transitive, we have $d(H,U_0)=d(H,T)=0$.
On the other hand, the assumption of the proposition yields that
\[d(H,U_1)=d(H,W) \geq \frac{k!}{\lvert\Aut(H)\rvert}\cdot 2^{-\binom{k}{2}}.\]
Since $d(H,U_{\alpha})$ is a continuous function of $\alpha\in [0,1]$,
there exists $\alpha\in (0,1]$ such that
\[d(H,U_{\alpha})=\frac{k!}{\lvert\Aut(H)\rvert}\cdot 2^{-\binom{k}{2}}.\]
\end{proof}

A classical result on tournaments of Beineke and Harary~\cite{BeiH65} on Tur\'an density of a cycle $C_3$ of length three translates to the language of tournament limits as follows:
$d(C_3,W)\le 1/4$ and the equality holds if and only if
\[\int_{[0,1]}W(x,y)\dd y=\frac{1}{2}\]
for almost every $x\in [0,1]$.
Hence, the cycle $C_3$ is not quasirandom-forcing by Proposition~\ref{prop:qforce}.
Let $C_4$ be the $4$-vertex tournament obtained from the cycle of length four by adding two edges (note that
all tournaments obtained in this way are isomorphic).
The result of~\cite{BeiH65} on the Tur\'an density of $C_4$, in the language of tournament limits,
asserts $d(C_4,W)\le 1/2$ and the equality can be attained.
Hence, the tournament $C_4$ is not quasirandom-forcing by Proposition~\ref{prop:qlarge}.

We next define a notion of a (weighted) step tournamenton,
which is analogous to the notion of a step graphon.
Informally speaking,
a step tournamenton represents a large tournament such that
its vertices can be split into a finite number of parts such that
the tournament is quasirandom within each part and between the parts.
The formal definition goes as follows.
A matrix $A$ is a \emph{tournament matrix} if it is a square matrix, say of order $k$, with non-negative entries such that
$A_{ij}+A_{ji}=1$ for all $i,j\in [k]$.
A vector $w$ is \emph{stochastic} if all its entries are non-negative and they sum to one.
Let $A$ be a $k\times k$ tournament matrix and $w$ a $k$-dimensional stochastic vector.
Further, let $V_1,\dots,V_k$ be a partition of $[0,1]$ into disjoint measurable sets such that
the measure of $V_i$ is $w_i$, $i\in [k]$.
We define a tournamenton $W[A,w]$ as
\[W[A,w](x,y)=A_{i,j}\]
for every $(x,y)\in (V_i,V_j)$. 
A tournamenton $W$ such that 
there exists a tournament matrix $A$ and a (positive) stochastic vector $w$ such that $W=W[A,w]$
is called a \emph{weighted step tournamenton}.
If $w_i=1/k$ for all $i\in [k]$, we simply write $W[A]$ instead of $W[A,w]$.
Finally, if $H$ is a tournament,
then the \emph{blow-up} of $H$ is the tournamenton $W[A]$
where $A$ is the adjacency matrix of $H$ with $1/2$ on its diagonal.

Observe that the following formula holds for the density of $H$ in $W[A,w]$:
\begin{equation}\label{eq:d(H,A)}
d(H,W[A,w])=\frac{1}{\lvert\Aut(H)\rvert}\sum_{f:V(H)\to [k]} \prod_{i \in V(H)} w_{f(i)} \prod_{\overrightarrow{v_iv_j}\in E(H)} A_{f(i),f(j)},
\end{equation}
where $k$ is the order of the matrix $A$.
The identity \eqref{eq:d(H,A)} leads us to define $d^*(H,A,w)$ as follows.
\begin{equation}\label{eq:d*(H,A)}
d^*(H,A,w)=\sum_{f:V(H) \to [k]}\prod_{i \in V(H)} w_{f(i)} \prod_{\overrightarrow{v_iv_j}\in E(H)} A_{f(i),f(j)}.
\end{equation}
Again,
if each entry of $w$ is equal to $1/k$, we will simply write $d^*(H,A)$ instead of $d^*(H,A,w)$.

By combining Proposition~\ref{prop:qlarge},
the definition of $d^*(H,A,w)$, and the identities \eqref{eq:d(H,W)} and \eqref{eq:d(H,A)},
we obtain the following.

\begin{proposition}
\label{prop:labelledineq}
Let $H$ be a $k$-vertex non-transitive tournament. 
If there exists an $\ell\times\ell$ tournament matrix $A$ and an $\ell$-dimensional positive stochastic vector $w$ such that
not all entries of $A$ are equal to $1/2$ and
\[ d^*(H,A,w) \geq 2^{-\binom{k}{2}},\]  
then $H$ is not quasirandom-forcing.
\end{proposition}

\section{General arguments}
\label{sec:general}

The purpose of this section is to establish the following two statements and
use them to show that most $6$-vertex tournaments are not quasirandom-forcing.

\begin{proposition}
\label{prop:sconnected}
Let $H$ be a non-transitive tournament.
If $H$ is not strongly connected, then $H$ is not quasirandom-forcing.
\end{proposition}

\begin{proposition}
\label{prop:twins}
Let $H$ be a non-transitive $6$-vertex tournament.
If $H$ contains twins or has a non-trivial automorphism,
then $H$ is not quasirandom-forcing.
\end{proposition}

\begin{proof}[Proof of Proposition~\ref{prop:sconnected}]
Let $k$ be the number of vertices of $H$. Note that $k\ge 4$.
For simplicity, we will write $\rho$ for $2^{-\binom{k}{2}}$.
Since the tournament $H$ is not strongly connected,
its vertices can be split into non-empty sets $X_1$ and $X_2$ such that
all edges are oriented from $X_1$ to $X_2$;
let $k_1$ and $k_2$ be the sizes of $X_1$ and $X_2$, respectively.
For each $\alpha \in [0,1]$,
consider the following tournament matrix and stochastic vector
\begin{equation*}
A=\begin{pmatrix}
1/2 & 1 \\
0 & 1/2 \\
\end{pmatrix}
\mbox{ and }
w=(\alpha, 1-\alpha),
\end{equation*}
and set $W_\alpha=W[A,w]$.
Our aim is to find an appropriate $\alpha\in(0,1)$ so that we can apply Proposition~\ref{prop:labelledineq} to $W_\alpha$.
Observe that
\begin{equation}\label{eq:1}
d^*(H,A,w) \geq \alpha^k \cdot \rho + \alpha^{k_1}(1 - \alpha)^{k_2} \cdot 2^{k_1 k_2} \cdot \rho + (1-\alpha)^k \cdot \rho.
\end{equation}
Note that the inequality is strict if $H$ has more than two strongly connected components.
We show that $d^*(H,A,w) > \rho$ for some $\alpha \in (0, 1)$ in each of the following cases.

If $k_1 = 1$, we use the second and third term of~(\ref{eq:1}) to lower bound $d^*(H,A,w)$ as follows:
\begin{align*}
   d^*(H,A,w) &>
   \alpha(1 - \alpha)^{k_2} \cdot \rho \cdot 2^{k_2} + (1-\alpha)^k \cdot \rho\\
   &= \rho + \alpha \cdot (2^{k_2} - k) \rho + O(\alpha^2). 
\end{align*}
Since $k \geq 4$,
it holds that $2^{k_2}-k>0$, and
we conclude that $d^*(H,A,w) > \rho$ for some positive $\alpha$ that is sufficiently small.
The case $k_1 = k - 1$ is symmetric to the case $k_1=1$.
Hence, it remains to analyze the case when $2\le k_1\le k-2$.

If $2\le k_1\le k-2$, we set $\alpha = 1/2$. It follows from~(\ref{eq:1}) that
\begin{equation*}
   d^*(H,A,w) \ge 2^{-k} \cdot \rho +
   2^{-k_1} \cdot 2^{-k_2} \cdot 2^{k_1k_2} \cdot \rho
   + 2^{-k} \cdot \rho
   \geq \left(1 + 2^{1-k}  \right) \cdot \rho,
\end{equation*}
where the last inequality holds since $k_1k_2\ge k_1+k_2$.
This concludes the proof.
\end{proof}

We prove Proposition~\ref{prop:twins} by an argument similar to that
used in~\cite{BucLSS19} to observe that
every quasirandom-forcing tournament with seven or more vertices is transitive.

\begin{proof}[Proof of Proposition~\ref{prop:twins}]
Let $A$ be the adjacency matrix of $H$ with $1/2$ on its diagonal.
If $H$ has a non-trivial automorphism, then $d^*(H,A) \ge 2\cdot 6^{-6} > 2^{-15}$
as there are at least two choices of $f$ in the sum in \eqref{eq:d*(H,A)} for which the expression in the definition is non-zero.
It follows that $H$ is not quasirandom-forcing by Proposition~\ref{prop:labelledineq}.

We now consider the case that $H$ has twins.
Let $v_1,\ldots,v_6$ be the vertices of $H$ and 
assume by symmetry that $v_1$ and $v_2$ are the twins.
As in the previous case, it is enough to show that $d^*(H,A) \ge 2\cdot 6^{-6}$.
This time, observe that the innermost product in \eqref{eq:d*(H,A)} is equal to one for the map $f$ with $f(v_i)=i$ for all $i \in [6]$, and
it is equal to $1/2$ for the two maps $f$ satisfying $f(v_1)\in\{1,2\}$, $f(v_1)=f(v_2)$ and $f(v_i)=i$, where $i\in\{3,4,5,6\}$.
\end{proof}

Proposition~\ref{prop:sconnected} implies that
every quasirandom-forcing non-transitive tournament $H$ is strongly connected.
The classical results on the Tur\'an density of $C_3$ and $C_4$ (see the discussion of these results in Section~\ref{sec:prelim})
yield that there is no such tournament $H$ with three or four vertices, and
the observation of Buci\'{c} et al.~\cite{BucLSS19} yield that
there are no such tournaments $H$ with seven or more vertices.
Hence, we are left to analyze tournaments with five and six vertices.
In the case of $5$-vertex tournaments,
the results of Coregliano et al.~\cite{CorPS19} imply that
all $5$-vertex strongly connected tournaments with the possible exception of two tournaments are not quasirandom-forcing.
The two exceptional tournaments are $F_5$, which is depicted in Figure~\ref{fig:F5} and is quasirandom-forcing, and
$H_5$, which is depicted in Figure~\ref{fig:5in7} and is shown to be not quasirandom-forcing in the next section.

There are $55$ non-transitive tournaments on $6$ vertices,
out of which $20$ are not strongly connected, $29$ contain twins, and $15$ have a non-trivial automorphism (some tournaments have more than one of these properties);
see Table~\ref{tab:vertex6}.
A SageMath~\cite{sagemath} script that verifies the entries of Table~\ref{tab:vertex6}
is available as an ancillary file on arXiv associated with the arXiv version of this manuscript~\cite{ancillary}.

By the discussion in the previous paragraph, Propositions~\ref{prop:sconnected} and~\ref{prop:twins} yield that $41$ non-transitive $6$-vertex tournaments are not quasirandom-forcing.
We will analyze the remaining $14$ tournaments, which are depicted in Figure~\ref{fig:H6}, in the next section.

\begin{table}
\begin{center}
\scalebox{0.8}{$
\begin{array}{|lll|ll|ll|l|lll|ll|ll|l}
\cline{1-7}
\cline{9-15}
\mbox{A} & \mbox{B} & \mbox{C} & \mbox{D} & \mbox{E} & \mbox{Tournament} &&& \mbox{A} & \mbox{B} & \mbox{C} & \mbox{D} & \mbox{E} & \mbox{Tournament} &&\\
\cline{1-7}
\cline{9-15}
\bullet  &  \bullet  &  \bullet  &         &         & 00000,0000,000,01,0  &  &&
         &           &           & \bullet &         & 00110,0001,000,01,0  &  H_6^{2}  \\
\bullet  &           &  \bullet  &         &         & 00010,0000,000,00,0  &  &&
\bullet  &           &           &         &         & 00100,0010,000,00,0  & \\
         &           &  \bullet  &         &         & 00011,0000,000,00,0  &  &&
         &           &           & \bullet &         & 00101,0010,000,00,0  &  H_6^{3}  \\
         &           &  \bullet  &         &         & 00010,0001,000,00,0  &  &&
         &           &  \bullet  &         &         & 00100,0011,000,00,0  & \\
         &           &           &         & \bullet & 00010,0000,001,00,0  &  H_6^{1}   &&
         &           &           & \bullet &         & 00100,0010,001,00,0  &  H_6^{4}  \\
         &           &  \bullet  &         &         & 00010,0000,000,01,0  &  &&
         &           &           & \bullet &         & 00100,0010,000,01,0  &  H_6^{5}  \\
         &           &  \bullet  &         &         & 00010,0000,000,00,1  &  &&
         &           &           &         & \bullet & 00100,0010,000,00,1  &  H_6^{6}  \\
\bullet  &           &  \bullet  &         &         & 00000,0010,000,00,0  &  &&
         &  \bullet  &           &         &         & 00101,0010,001,00,0  & \\
         &           &  \bullet  &         &         & 00001,0010,000,00,0  &  &&
         &           &           &         & \bullet & 00100,0011,001,00,0  &  H_6^{7}  \\
\bullet  &           &  \bullet  &         &         & 00000,0011,000,00,0  &  &&
         &           &           & \bullet &         & 00100,0011,000,01,0  &  H_6^{8}  \\
\bullet  &           &           &         &         & 00000,0010,001,00,0  &  &&
\bullet  &  \bullet  &           &         &         & 00110,0010,000,00,0  & \\
\bullet  &           &           &         &         & 00000,0010,000,01,0  &  &&
         &           &           &         & \bullet & 00111,0010,000,00,0  &  H_6^{9}  \\
\bullet  &           &  \bullet  &         &         & 00000,0010,000,00,1  &  &&
         &           &  \bullet  &         &         & 00111,0011,000,00,0  & \\
\bullet  &  \bullet  &           &         &         & 00000,0011,001,00,0  &  &&
         &           &           & \bullet &         & 00111,0010,001,00,0  &  H_6^{10} \\
\bullet  &  \bullet  &  \bullet  &         &         & 00000,0000,010,00,0  &  &&
\bullet  &  \bullet  &  \bullet  &         &         & 00000,0100,000,00,0  & \\
         &  \bullet  &           &         &         & 00001,0000,010,00,0  &  &&
\bullet  &  \bullet  &           &         &         & 00010,0100,000,00,0  & \\
\bullet  &  \bullet  &           &         &         & 00000,0001,010,00,0  &  &&
         &  \bullet  &  \bullet  &         &         & 00011,0100,000,00,0  & \\
\bullet  &           &  \bullet  &         &         & 00000,0000,011,00,0  &  &&
         &           &           & \bullet &         & 00010,0101,000,00,0  &  H_6^{11} \\
\bullet  &           &  \bullet  &         &         & 00100,0000,000,00,0  &  &&
         &  \bullet  &           &         &         & 00010,0100,000,00,1  & \\
\bullet  &           &  \bullet  &         &         & 00110,0000,000,00,0  &  &&
\bullet  &  \bullet  &  \bullet  &         &         & 01000,0000,000,00,0  & \\
         &           &  \bullet  &         &         & 00111,0000,000,00,0  &  &&
\bullet  &  \bullet  &           &         &         & 01000,0000,000,01,0  & \\
         &           &  \bullet  &         &         & 00110,0001,000,00,0  &  &&
\bullet  &           &           &         &         & 01010,0000,000,00,0  & \\
         &           &  \bullet  &         &         & 00110,0000,001,00,0  &  &&
         &           &  \bullet  &         &         & 01011,0000,000,00,0  & \\
         &           &  \bullet  &         &         & 00110,0000,000,01,0  &  &&
         &           &           & \bullet &         & 01010,0001,000,00,0  &  H_6^{12} \\
         &           &  \bullet  &         &         & 00110,0000,000,00,1  &  &&
         &           &           & \bullet &         & 01010,0000,001,00,0  &  H_6^{13} \\
         &           &  \bullet  &         &         & 00111,0000,001,00,0  &  &&
         &           &           &         & \bullet & 01010,0000,000,01,0  &  H_6^{14} \\
         &  \bullet  &  \bullet  &         &         & 00110,0001,001,00,0  &  &&
         &           &  \bullet  &         &         & 01010,0000,000,00,1  & \\
         &  \bullet  &  \bullet  &         &         & 00111,0000,000,01,0  &  &&&&&&&&\\
\cline{1-7}
\cline{9-15}
\end{array}
$}
\end{center}
\caption{The table indicates for each 6-vertex non-transitive tournament the way in which it was shown to be not quasirandom-forcing as follows. A:~by~Proposition~\ref{prop:sconnected} because it is not strongly connected, B:~by~Proposition~\ref{prop:twins} because it has a non-trivial automorphism, C:~by~Proposition~\ref{prop:twins} because it has twins, D:~Subsection~\ref{subs:arg1}, and E:~Subsection~\ref{subs:arg2}. The tournaments are described by the upper-triangle part of their adjacency matrix, see the beginning of Section~\ref{sec:constructions}, and by the notation used for the tournament if a specific notation has been introduced.}
\label{tab:vertex6}
\end{table}

\section{Specific constructions}
\label{sec:constructions}

In this section we provide two different types of arguments
to rule out the remaining $15$ tournaments from being quasirandom-forcing.
Tournaments that we consider will be described by the upper-triangle part of their adjacency matrix,
i.e., if $A$ is the adjaceny matrix of a $k$-vertex tournament,
then the tournament is described by
\[[A_{1,2}\cdots A_{1,k},A_{2,3}\cdots A_{2,k},\ldots,A_{k-2,k-1}A_{k-2,k},A_{k-1,k}].\]
The remaining $5$-vertex tournament, which is depicted in Figure~\ref{fig:5in7},
is described by $[0010,001,00,0]$,
We denote this tournament $H_5$ (this tournament is called $T^{10}_{5}$ in~\cite{CorPS19}).
The $14$ remaining $6$-vertex tournaments,
which can also be found in Figure~\ref{fig:H6},
are the following:
\begin{align*}
 H^{1}_6: [00010,0000,001,00,0], & \quad  H^{2}_6: [00110,0001,000,01,0], \\
 H^{3}_6: [00101,0010,000,00,0], & \quad  H^{4}_6: [00100,0010,001,00,0], \\
 H^{5}_6: [00100,0010,000,01,0], & \quad  H^{6}_6: [00100,0010,000,00,1], \\
 H^{7}_6: [00100,0011,001,00,0], & \quad  H^{8}_6: [00100,0011,000,01,0], \\
 H^{9}_6: [00111,0010,000,00,0], & \quad H^{10}_6: [00111,0010,001,00,0], \\
H^{11}_6: [00010,0101,000,00,0], & \quad H^{12}_6: [01010,0001,000,00,0], \\
H^{13}_6: [01010,0000,001,00,0], & \quad H^{14}_6: [01010,0000,000,01,0]. 
\end{align*}

\begin{figure}
\begin{center}
\epsfbox{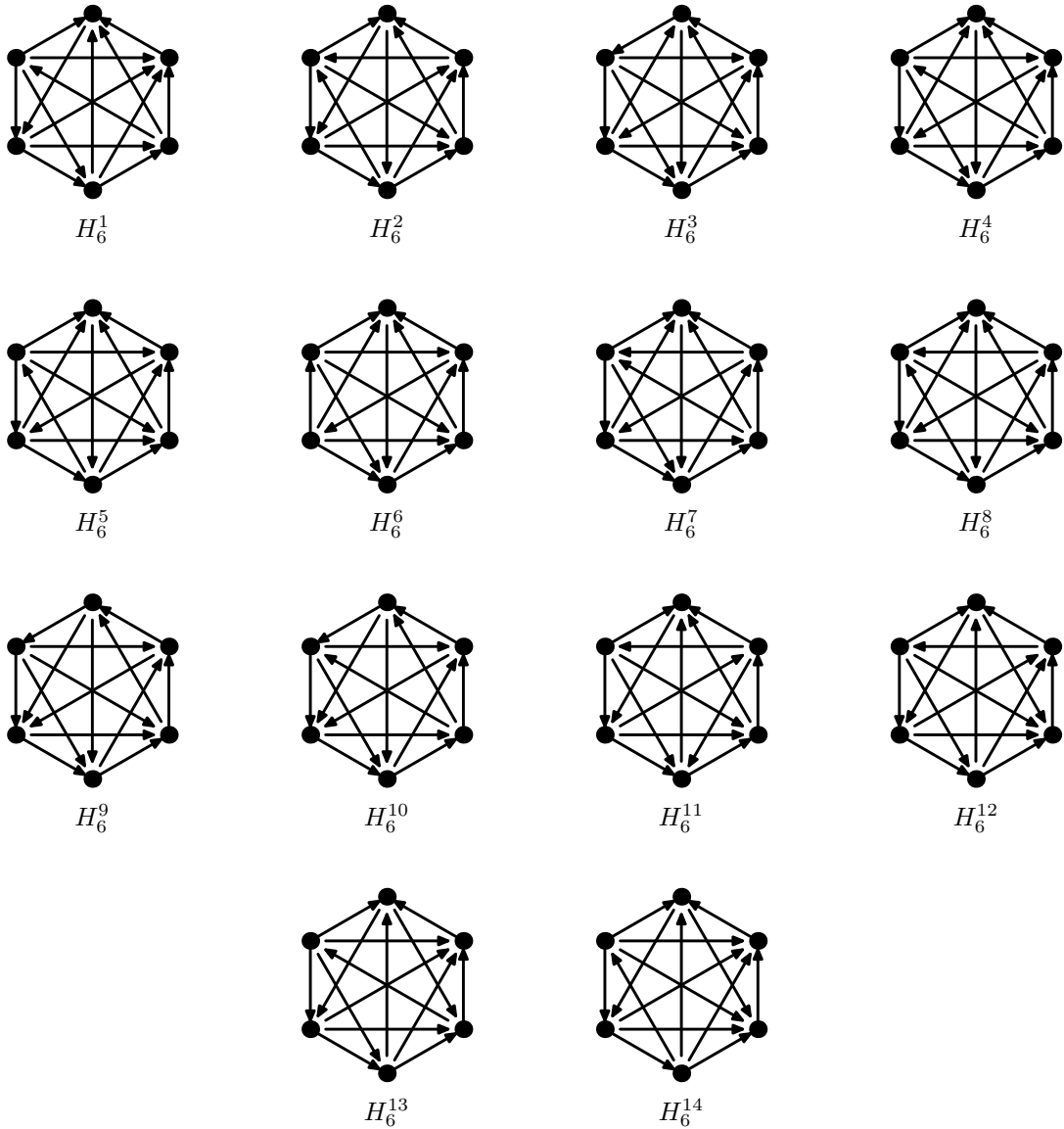}
\end{center}
\caption{The tournaments $H^{1}_6, \dots,  H^{14}_6$.}
\label{fig:H6}
\end{figure}

\subsection{Blow-ups}
\label{subs:arg1}

We start this subsection with the following statement,
which can also be found in~\cite{BucLSS19}.
Let $n(H,S)$ be the number of copies of a tournament $H$ in a tournament $S$,
i.e., $n(H,S)=d(H,S)\cdot\binom{\lvert H\rvert}{\lvert S\rvert}$.

\begin{proposition}
\label{prop:targets}
Let $H$ be a non-transitive $k$-vertex tournament.
If there exists an $s$-vertex tournament $S$, $s>k$, such that $n(H,S) \geq s^k\cdot 2^{-\binom{k}{2}}$, 
then $H$ is not quasirandom-forcing.
\end{proposition}

\begin{proof}
Let $A$ be the adjacency matrix of $H$ with $1/2$ on its diagonal.
Note that $d^*(H,A)\ge n(H,S)\cdot s^{-k}$.
Since $n(H,S)\ge s^k\cdot 2^{-\binom{k}{2}}$,
Proposition~\ref{prop:labelledineq} yields that $H$ is not quasirandom-forcing.
\end{proof}

We consider tournaments $S_7$, $S_{11}$ and $S_{15}$ with $7$, $11$ and $15$ vertices, respectively;
we remark that the tournaments $S_{11}$ and $S_{15}$ have been identified by a heuristic computer search
maximizing the number of copies of tournaments $H^i_6$.
\begin{align*}
S_7: & [001011,00101,0010,001,00,0], \\
S_{11}: & [1100110001,101001011,11010101, 0001101,100011, 00110,1000,100,10,0], \\
S_{15}: & [01010100100110,0011110000001,010001001101,10011000010,1011101010, \\
& 110110010,11101001,1110001,010110,11110,0101,001,10,0].
\end{align*}
The tournament $S_7$ is depicted in Figure~\ref{fig:5in7}.
It is interesting to note that $n(H_5,S_7)=21$,
i.e., every $5$-tuple of vertices of $S_7$ induces $H_5$, and
the tournaments $S_{7}$ and $S_{11}$ are Paley tournaments~\cite{ErdR63,GraS71,SacH62}.
In particular, the adjacency matrix of $S_7$ is the incidence matrix of the points and lines of the Fano plane.
Since $n(H_5,S_7)=21$,
Proposition~\ref{prop:targets} implies that $H_5$ is not quasirandom-forcing.
\begin{figure}
\begin{center}
\epsfbox{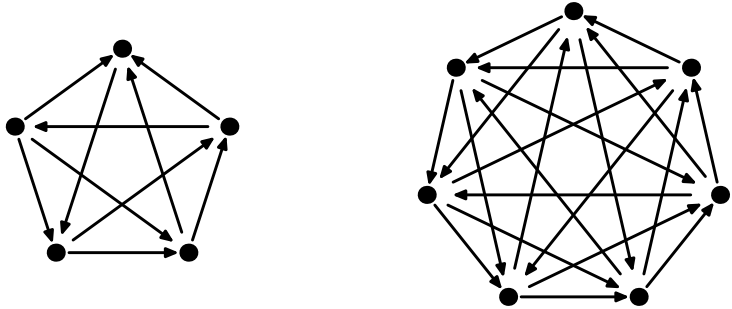}
\end{center}
\caption{The tournaments $H_5$ and $S_7$.}
\label{fig:5in7}
\end{figure}
It also holds that
$n(H^{i}_6,S_{11})=55$ for $i \in \{2,3,4,8,10,11,13\}$ and
$n(H^{i}_6,S_{15})=357$ for $i \in \{5,12\}$.
Proposition~\ref{prop:targets} implies that
none of the tournaments $H^i_6$, $i\in\{2,3,4,5,8,10,11,12,13\}$,
are quasirandom-forcing.

\subsection{Step tournamentons with variable weights}
\label{subs:arg2}

It remains to analyze the tournaments $H^i_6$ for $i\in\{1,6,7,9,14\}$.
We consider the following three tournament matrices, each of which is a function of $x\in [-1/2,1/2]$, and
show that there exists $x\not=0$ such that Proposition~\ref{prop:labelledineq} can be applied.\begin{align*}
A_x &=
\begin{pmatrix}
1/2 & 1/2-x  \\
1/2+x & 1/2  \\
\end{pmatrix}, \\
B_x &=
\begin{pmatrix}
1/2 & 1/2-x & 1/2+x \\
1/2+x & 1/2 & 1/2-x \\
1/2-x & 1/2+x & 1/2 \\
\end{pmatrix}, \\
C_x &=
\begin{pmatrix}
1/2 & 1/2-x & 1/2+x & 1/2-x \\
1/2+x & 1/2 & 1/2-x & 1/2-x \\
1/2-x & 1/2+x & 1/2 & 1/2-x \\
1/2+x & 1/2+x & 1/2+x & 1/2 \\
\end{pmatrix}.
\end{align*}
We next compute the densities of $H_6^{14}$, $H_6^9$ and $H_6^6$.
\begin{align*}
d^*(H_6^{14},A_x)&=
 \frac{1}{32768} + \frac{x^2}{8192} - \frac{5x^4}{16384} - \frac{9x^6}{4096} - \frac{7x^8}{4096},\\
d^*(H_6^{9},B_{x})&=
\frac{1}{32768} + \frac{x^4}{3072} -\frac{x^6}{216} -\frac{5x^8}{5184} + \frac{13x^{10}}{486} - \frac{x^{12}}{324},\\
d^*(H_6^{6},C_{x})&=
 \frac{1}{32768} + \frac{3x^3}{32768} -\frac{81x^4}{131072} - \frac{3x^5}{8192} 
 + \frac{27x^6}{65536} - \frac{63x^8}{131072} + \frac{15x^{12}}{1024}.
\end{align*}
The maximum of each of the three polynomials above is larger than $2^{-15}\approx 0.000030518$.
In particular, the first one is larger than $0.000037337$ for $x=0.30721$,
the second is larger than $0.000030757$ for $x=0.21740$, and
the third is larger than $0.000030544$ for $x=0.10418$.
Hence, Proposition~\ref{prop:labelledineq} yields that
none of the tournaments $H_6^{14}$, $H_6^{9}$ and $H_6^{6}$ are quasirandom-forcing.
Since the tournament $H_6^{7}$ can be obtained from $H_6^{9}$ by reversing the orientation of all its edges,
it follows that $d^*(H_6^{9},B_{x})=d^*(H_6^{7},B_{-x})$.
Similarly, the tournament $H_6^{1}$ can be obtained from $H_6^{6}$ by reversing the orientation of all its edges and
$d^*(H_6^{6},C_{x})=d^*(H_6^{1},C_{-x})$.
Hence, $d^*(H_6^{7},B_{-x})>2^{-15}$ for $x=0.21740$ and $d^*(H_6^{1},C_{-x})>2^{15}$ for $x=0.10418$, and
neither $H_6^{7}$ nor $H_6^{1}$ is quasirandom-forcing by Proposition~\ref{prop:labelledineq}.

\bibliographystyle{bibstyle}
\bibliography{qtourn}

\Addresses

\end{document}